\documentclass{amsart}
\usepackage{mathrsfs}
\usepackage{amssymb,latexsym,amsmath,amsthm,enumitem,mathrsfs}
\usepackage{hyperref}

\usepackage{color}
\usepackage{stmaryrd}
\usepackage{mathrsfs}
\usepackage{lineno}
\usepackage{bbm}
\usepackage{scalerel,stackengine}
\usepackage{tikz}

\usepackage{ esint }
\usepackage{graphicx}
\usepackage{amssymb}
\newtheorem{theorem}{Theorem}[section]
\newtheorem{lemma}[theorem]{Lemma}
\newtheorem{remark}[theorem]{Remark}
\newtheorem{proposition}[theorem]{Proposition}
\newtheorem{question}{Question}
\newtheorem{definition}[theorem]{Definition}

\newtheorem{conjecture}{Conjecture}

\theoremstyle{definition}

\newcommand{\cA}{\mathcal{A}}
\newcommand{\cB}{\mathcal{B}}

\newcommand{\cG}{\mathcal{G}}
\newcommand{\cH}{\mathcal{H}}

\newcommand{\cL}{\mathcal{L}}
\newcommand{\cM}{\mathcal{M}}
\newcommand{\cN}{\mathcal{N}}

\newcommand{\cU}{\mathcal{U}}
\newcommand{\cV}{\mathcal{V}}
\newcommand{\cW}{\mathcal{W}}

\newcommand{\bA}{\mathbf{A}}

\newcommand{\M}{\mathbb{M}}
\newcommand{\bM}{\boldsymbol{M}}

\newcommand{\N}{\mathbb N}

\newcommand{\e}{\varepsilon}

\renewcommand{\L}{\mathbb L}

\newcommand{\cQ}{\mathcal{Q}}

\def\set4{\mathcal I}
\def\tup14{(1,2,3,4)}

\def\be{{\mathbf e}}

\def\bw{{\mathbf w}}

\newtheorem{example}{Example}

\newcommand{\R}{\mathbb{R}}
\newcommand{\D}{\mathbb D}
\newcommand{\de}{\delta} 
\newcommand{\ga}{\gamma}

\newcommand{\wt}{\widetilde}

\newcommand{\BL}{\textup{BL}}
\newcommand{\dir}{\textup{dir}}
\newcommand{\tx}{\textup}

\newcommand{\bV}{\boldsymbol{V}}
\newcommand{\Atr}{A_{\textup{trans}}}
\newcommand{\Aloc}{A_{\textup{loc}}}
\newcommand{\vbw}{\vec\bw}

\begin{document}

 \author{Shengwen Gan}
 \address{Department of Mathematics\\
 University of Wisconsin-Madison, WI 53706, USA}
 \email{sgan7@wisc.edu}


\date{}

\title[Kakeya problem and projection problem in Grassmannian]{Kakeya problem and projection problem for $k$-geodesics in Grassmannians}

\begin{abstract}
The Kakeya problem in $\R^n$ is about estimating the size of union of $k$-planes; the projection problem in $\R^n$ is about estimating the size of projection of a set onto every $k$-plane ($1\le k\le n-1$). The $k=1$ case has been studied on general manifolds in which $1$-planes become geodesics, while $k\ge 2$ cases were still only considered in $\R^n$. We formulate these problems on homogeneous spaces, where $k$-planes are replaced by $k$-dimensional totally geodesic submanifolds. After formulating the problem, we prove a sharp estimate for Grassmannians.
\end{abstract}

\maketitle

\section{Introduction}

The goal of this paper is to study the Kakeya problem and projection problem for $k$-geodesics in homogeneous spaces. Here, $k$-geodesic refers to $k$-dimensional totally geodesic submanifold. We begin with the introduction of Kakeya problem and projection problem in $\R^n$.

Kakeya problem is one of the famous problem in harmonic analysis. We first state its original line-version as follows.
\begin{conjecture}[Kakeya conjecture]
    If a set $K\subset \R^n$ contains a line in each direction, then $\dim(K)=n$. (Here and throughout the paper, $\dim$ will denote the Hausdorff dimension.) 
\end{conjecture}

Bourgain also considered the $k$-plane version, which is known as the $(k,d)$-problem. See \cite{bourga1991besicovitch}.

The original conjecture is very hard. However, if we consider a variant of the Kakeya conjecture by removing the condition that the lines point in different directions, then we have the complete answer. See the following result proved by the author in \cite{gan2023hausdorff}. 

\begin{theorem}\label{unionofkplane}
Let $0<k\le d<n$ be integers and $\beta\in [0,k+1]$. Let $\cV\subset A(k,n)$ be a set of $k$-planes in $\R^n$, with $\dim(\cV)= (k+1)(d-k)+\beta$. Then,
\begin{equation}\label{mainbound}
    \dim(\bigcup_{V\in\cV}V)\ge d+\min\{1,\beta\}.
\end{equation}  

\end{theorem}

\begin{remark}
    {\rm
    We give some historical remarks about this problem. D. Oberlin \cite{oberlin2011exceptional} initiated the study of this problem. Later, H\'era formulated the conjecture (see Conjecture 1.16 in \cite{hera2018hausdorff}). By using the multilinear Kakaya, Zahl \cite{zahl2022unions} proved the problem for union of lines. Finally, based on the idea of Zahl, the author completely solved the problem for all dimensions.
    }
\end{remark}

It is very natural to replace the ambient space $\R^n$ by a general Riemannian manifold, and consider the Kakeya problem for $k$-planes in a manifold. The first question is: what are $k$-planes in a manifold? When $k=1$, the natural analog are the geodesics in a manifold. Hence, the line-version Kakeya problem in a manifold is about estimating the size of unions of geodesics. For general $k$, to author's knowledge, it did not appear in any previous reference. In this paper, we will suggest a way to formulate this problem.

Naturally, the counterpart for $k$-planes in a manifold should be $k$-dimensional totally geodesic submanifolds. However, not all manifolds have $k$-dimensional totally geodesic submanifolds. Another thing we are concerned about is the behavior of $k$-dimensional totally geodesic submanifolds at every point. In $\R^n$, we know that for every point, the set of $k$-planes through the point can be identified with $G(k,n)$, which is the same for every point. However, this may not be true for a general manifolds. These two issues limit us to focus on a smaller range of manifolds which have many totally geodesic submanifolds and have many symmetries, for example, the homogeneous spaces.

Let $\cM=G/H$ be a compact homogeneous space where $G$ and $H$ are compact Lie groups. One of the most important examples of the compact homogeneous spaces is the Grassmannian.
    \[G(l,n)\cong O(n)/O(l)\times O(n-l).\]
If $\cM=G/H$ is a homogeneous space, then there exists a Riemannian metric on $\cM$ that is invariant under the group action by $G$. Therefore, $\cM$ is naturally equipped with a Riemannian metric and hence is a Riemannian manifold. We can talk about its totally geodesic submanifolds. 
\begin{definition}
Let $\cM$ be a homogeneous space. We define $\cA_k(\cM)$ to be the set of $k$-dimensional, complete, totally geodesic submanifolds of $\cM$.
    For $x\in \cM$, we define $\cG_k(x,\cM)$ to be the set of $k$-dimensional, complete, totally geodesic submanifolds of $\cM$ that contain $x$.
\end{definition}

For convenience, we just call ``$k$-dimensional, complete, totally geodesic submanifold" as ``$k$-geodesic". With the standard metric on $\R^n$, we see that $\cA_k(\R^n)=A(k,n)$ and $ \cG_k(x,\R^n)=x+G(k,n)$, which is the reason for our notation.
We also see that $\cA_k(S^n)$ is the set of $k$-dimensional great spheres, and $\cG_k(x,S^n)$ is the set of $k$-dimensional great spheres that contain $x$.
     Since the Riemannian metric is invariant under the group action, we have $\cG_k(x,\cM)=(\wt x\wt y^{-1})\cdot\cG_k(y,\cM)$, where $\wt x$ $\wt y\in G$ are elements in the cosets $x$ and $y$.  In other words, for $\cN\in \cG_k(x,\cM)$, the left multiplication gives a submanifold $(\wt x\wt y^{-1})\cdot \cN$ in $\cG_k(y,\cM).$
     
     In most cases, there is a natural parametrization on $\cA_k(\cM)$ so that we can view $\cA_k(\cM)$ as a Riemannian manifold and talk about the Hausdorff dimension of its subsets. The same is true for $\cG_k(x,\cM)$.

\subsection{Kakeya problem for \texorpdfstring{$k$}{}-geodesics }

Motivated by Theorem \ref{unionofkplane}, we ask the general question.

\begin{question}\label{Q1}
    Let $\cM$ be a homogeneous space. Fix an integer $1\le k<\dim(\cM)$. Let $\cV\subset \cA_k(\cM)$. How to estimate
    \begin{equation}
        \dim(\bigcup_{\cN\in \cV}\cN)\ ?
    \end{equation}
\end{question}

Let us focus on a specific example $A(l,n)$, which is a generic open subset of the compact homogeneous space $G(l+1,n+1)$. We identify $A(l,n)$ as an open subset of $G(l+1,n+1)$ in the following way. For any $L\in G(l+1,n+1)$ that is not parallel to $\R^n\times \{1\}$, we see that $(\R^n\times \{1\})\cap L$ is an $l$-dimensional plane in $\R^n\times \{1\}$. This correspondence $L\mapsto (\R^n\times \{1\})\cap L$ identifies a generic open subset of $G(l+1,n+1)$ with $A(l,\R^n\times \{1\})\cong A(l,n)$.

The reason that $A(l,n)$ is an interesting object to study is because $A(l,n)$ is one of the important homogeneous spaces. Also, when $l=0$, we have $A(0,n)=\R^n$, which recovers the most standard space, the Euclidean space.

Let us talk about the totally geodesic submanifolds.
In this paper, we are not going to consider all the $k$-geodesics, but just those that naturally appear. Let us describe them. Let $l\le m\le n$, and $V$ be an $m$-dimensional plane in $\R^n$. Then $A(l,V)$, the set of $l$-dimensional planes in $V$, can be viewed as a subset of $A(l,n)$. Moreover, $A(l,V)$ is an $(l+1)(m-l)$-geodesic in $A(l,n)$. We give a heuristic proof here (rigorous proof is in Appendix Lemma \ref{lemm1}, Lemma \ref{lemmm1}). Let $L_0,L_1\in A(l,V)$ and $\ga(t) (t\in[0,1])$ be a geodesic in $A(l,n)$ connecting $L_0, L_1$. Intuitively, the shortest path from $L_0$ to $L_1$ should lie in $\textup{span}(L_0,L_1)\subset V$. This means that $\ga\subset A(l,V)$, so $A(l,V)$ is totally geodesic.

The main task of the paper is to study Question \ref{Q1} for $\cM=A(l,n)$, $k=(l+1)(m-l)$ with $l\le m\le n$, $\cV$ restricted to the subset of $\{A(l,V): V\in A(m,n)\}$. It will be convenient for us to notice that $\{A(l,V): V\in A(m,n)\}$ is naturally parametrized by $A(m,n)$. 

\begin{theorem}[Main theorem]\label{mainthm}
    Fix integers $0\le l\le m\le d< n$ and $\beta\in [0,m+1]$. 
    Let $\cV\subset A(m,n)$ be a Borel set with $\dim(\cV)=(m+1)(d-m)+\beta$. Then, the set $\bigcup_{V\in\cV}A(l,V)$, as a subset of $A(l,n)$, has Hausdorff dimension
    \begin{equation}
        \dim(\bigcup_{V\in\cV}A(l,V))\ge (l+1)(d-l)+\min\{l+1,\beta\}.
    \end{equation}
\end{theorem}

\begin{remark}
    { \rm The theorem is sharp. The sharpness is showed in Section \ref{sec2}. One easily sees that Theorem \ref{mainthm} covers Theorem \ref{unionofkplane} at the special case $l=0$. }
\end{remark}

\bigskip

\subsection{Projection problem for \texorpdfstring{$k$}{}-geodesics}

The projection theory attracts a lot of research interests recently. In this subsection, we want to suggest a way to formulate projection problem in homogeneous spaces, and show how a recent result of the author \cite{gan2023marstrand} is related to the projection problem in $A(l,n)$. We are not going to prove any other technical result about projection problem in this paper.

We want to mention one progress made in this area recently, which is called the restricted projection problem.

Let $\ga(\theta)=(\cos\theta,\sin\theta,1)$, $\theta\in[0,2\pi)$. Define the planes
\begin{equation}\label{defplane}
    V_\theta:=\ga(\theta)^\perp, 
\end{equation} 
and lines
\begin{equation}\label{defline}
    \ell_\theta:=\textup{span}\{\ga(\theta)\}. 
\end{equation} 

We let $\pi_\theta:\R^3\rightarrow V_\theta$ and $\rho_\theta:\R^3\rightarrow \ell_\theta$ be orthogonal projections.
The following results were proved.

\begin{theorem}
    Let $A$ be a Borel set in $\R^3$. Then
    \begin{enumerate}
        \item \textup{(Restricted projections to planes)}  \[\dim(\pi_\theta(A))=\min\{2,\dim(A)\}, \textup{~for~a.e.~} \theta\in[0,2\pi)\]

        \item \textup{(Restricted projections to lines)} \[\dim(\rho_\theta(A))=\min\{1,\dim(A)\}, \textup{~for~a.e.~} \theta\in[0,2\pi)\]
    \end{enumerate} 
\end{theorem}

This problem was conjectured in \cite[Conjecture 1.6]{fassler2014restricted}. Later, \textit{(1)} was proved in \cite{gan2023restriction}, \textit{(2)} was proved in \cite{pramanik2022furstenberg} and \cite{gan2024exceptional}. After that, a more general result was derived for sets in $\R^n$ in \cite{gan2022restricted}.

We would also like to consider the projection problems in homogeneous spaces. 

\begin{definition}
    Let $\cM$ be a Riemannian manifold and $\cN$ be a submanifold. We define the map $\pi_{\cN}: \cM\rightarrow \cN$ as follows: for $x\in \cM$, then $y=\pi_\cN(x)$ is a point in $\cN$ such that
    \[ d(x,\cN)=d(x,y). \] 
    Here, $d(\cdot,\cdot)$ is the distance function induced by the Riemannian metric.
\end{definition}
$\pi_\cN(x)$ may not be uniquely defined, but when $x$ is close to $\cN$, $\pi_\cN(x)$ is uniquely defined. This will not cause any trouble, as we always consider the local problem in which $x$ and $\cN$ are close to each other.
\begin{example}
    Let $\cM=S^n(\subset \R^{n+1})$ and $\cN=S^{n-1}=S^n\cap(\R^{n}\times \{0\})$. If $x$ is the north pole of $S^n$, then for any $y\in S^{n-1}$ we have $d(x,S^{n-1})=d(x,y)$, so $\pi_\cN(x)$ can be any point in $\cN$.
\end{example}



The projection problem in homogeneous spaces can be formulated in the following way.
\begin{question}[Restricted Projection]\label{Q4}
    Let $\cM$ be a homogeneous space. Fix an integer $1\le k<\dim(\cM)$. Let $ \Theta \subset \cA_k(\cM)$, which is called the restricted set of directions. (Usually, $\Theta$ is chosen to be a submanifold of $\cA_k(\cM)$.)
    Let $\bA\subset \cM$. Find the optimal number $S=S_{\bA,\Theta}$ so that the following Marstrand-type estimate holds.
    \begin{equation}
        \dim(\pi_\cN(\bA))\ge S_{\bA,\Theta}, \ \textup{~for~a.e.~}\cN\in \Theta. 
    \end{equation}
In most cases, we can even expect the lower bound $S=S_{\dim(\bA),\Theta}$ to depend on $\dim(\bA)$ and $ \Theta$. 
\end{question}

In a recent paper of the author (see \cite[Theorem 2]{gan2023marstrand}), a special case of Question \ref{Q4} was studied, in which $\cM=A(1,n),k=2(m-1)$ with $1<  m< n$ and $\Theta=\{A(l,\Pi): \Pi\in G(m,n)\}$. We explain more.

Note that $A(l,\Pi)$ is totally geodesic in $A(l,n)$. Consider the projection
\[ \pi_{A(l,\Pi)}: A(l,n)\rightarrow A(l,\Pi). \]
Actually, we have the following way to characterize this projection. For the proof, see Lemma \ref{lemmm2} in Appendix.
\begin{lemma}
Let $\Pi\in G(m,n)$, where $l\le m\le n$. Then for any $L\in A(l,n)$, $\pi_{\Pi}(L)\subset\pi_{A(l,\Pi)}(L)$. In particular, if $\pi_\Pi(L)$ is $l$-dimensional, then $\pi_{\Pi}(L)=\pi_{A(l,\Pi)}(L)$.
\end{lemma}

We can reinterpret Theorem 2 from \cite{gan2023marstrand} as follows:

\begin{theorem}
Fix integers $1< m<n$. Let $\bA\subset A(1,n)$. Then
\[ \dim(\pi_{A(1,\Pi)}(\bA))\ge S(a), \ \textup{~for~a.e.~}\Pi\in G(m,n).  \]
Here, $S(a)$ is defined as follows. For $j=0$ or $l$, 
    \begin{align}
    &S(a)=\min\{a,m-1\},\ \ \ a\in [0,n-1].\\
    &S(a)=\min\{a-(n-m),2(m-1)\}, \ \ \ a\in [n-1,2(n-1)].
    \end{align}
\end{theorem}

\begin{remark}
    {\rm This theorem is sharp. Yet, the author do not know how to prove a sharp result for $A(l,n)$ with general $l$.
    }
\end{remark}

\subsection{Key ideas in the proof}

We roughly follow the steps as in \cite{gan2023hausdorff}, where the key tools are Brascamp-Lieb inequality and broad-narrow method. The broad-narrow method, also called Bourgain-Guth method, was developed by Bourgain and Guth in \cite{bourgain2010bounds} to study restriction type problems in harmonic analysis. After that, broad-narrow method became a key tool in harmonic analysis. The originality of this paper is that we will apply Brascamp-Lieb inequality and conduct broad-narrow argument in $A(l,n)$ rather than $\R^n$. 
The author believe for any homogeneous space, it is possible to find some sort of Brascamp-Lieb inequality and broad-narrow argument on it, which may have potential applications to other problems.

\subsection{Structure of the paper}

In Section \ref{sec2}, we construct sharp examples. In Section \ref{sec3}, we introduce the geometry of $A(l,n)$ that we need. In Section \ref{sec4}, we prove Theorem \ref{kakeyathm}. In Section \ref{sec5}, we show that Theorem \ref{kakeyathm} implies Theorem \ref{mainthm}.

\section{Sharp examples}\label{sec2}

We show Theorem \ref{mainthm} is sharp, i.e., we will find $\cV\subset A(m,n)$ with $\dim(\cV)=(m+1)(d-m)+\beta$, so that
\[ \dim(\bigcup_{V\in\cV}A(l,V))\le (l+1)(d-l)+\min\{l+1,\beta\}. \]
Our examples $V\in\cV$ will lie in $\R^{d+1}(\subset \R^n)$. 

\begin{itemize}
    \item When $\beta\in[l+1,m+1]$, let $\cV$ be a $((m+1)(d-m)+\beta)$-dimensional subset of $A(m,\R^{d+1})$. This is allowable since $\dim(A(m,d+1))=(m+1)(d+1-m)\ge (m+1)(d-m)+\beta$. Since every $V\in \cV$ is contained in $\R^{d+1}$, we have $\dim(\bigcup_{V\in\cV}A(l,V))\le \dim(A(l,\R^{d+1}))=(l+1)(d+1-l)$.

    \item When $\beta\in [0,l+1]$, we first talk about some geometry. Let $\Pi_0:=\R^{n-l}\times \{0\}^l$. Let $\{\be_i\}_{i=1}^n$ be the standard basis in $\R^n$.
    And let
    \[ \Pi_j:=\Pi_0+\vec\be_{n-l+j}, \text{ for } j=1,\dots,l. \]
    We see that $\Pi_0,\Pi_1,\dots,\Pi_j$ are parallel $(n-l)$-planes. 
    
    We also note that any $(l+1)$-tuple of points $(x_0,x_1,\dots,x_{l})$ with $x_j\in \Pi_j$ corresponds an $l$-plane $L$ that contains these $x_j$. We denote this $l$-plane by $L(x_0,x_1,\dots,x_{l})$. If $V_m$ is an $m$-plane and $V_m$ intersects each $\Pi_j$ transversely, then $V_m\cap \Pi_j$ are parallel $(m-l)$-planes for all the $j$; conversely, if $(v_0,v_1,\dots,v_{l})$ is an $(l+1)$-tuple where each $v_j$ is an $(m-l)$-plane in $\Pi_j$ and these $v_j$'s are parallel, then it corresponds to an $m$-plane $V_m=V_m(v_0,\dots,v_{l})$ so that $v_j= V_m\cap \Pi_j$ for all the $j$. The argument also holds when $m$ is replaced by $d$.

    Let $I_0:=\R^{n-d}\times \{0\}^{d-l}\times \{0\}^l$ and $J_0:=\{0\}^{n-d}\times \R^{d-l}\times \{0\}^l$. We see that $\Pi_0=I_0\oplus J_0$. For $j=1,\dots,l$, we also let 
    \[I_j:=I_0+\vec\be_{n-l+j},\ J_j:=J_0+\vec\be_{n-l+j}.\]

    If $(y_0,y_1,\dots,y_{l})$ is an $(l+1)$-tuple where each $y_j$ lies in $I_j$, then $y_j+J_0$ are parallel $(d-l)$-planes each of which lies in $\Pi_j$. Denote 
    $v_j=y_j+J_0$. Then, we obtain a $d$-plane $V_d=V_d(v_1,\dots,v_{l+1})$. Let $A_{\text{trans}}(m,V_d)$ be the set of $m$-planes that are contained in $V_d$ and are transverse to $\Pi_0$. In other words, $V_m\in A_{\text{trans}}(m,V_d)$ has the form
    \[ V_m(y_1+J',\dots,y_{l+1}+J'), \]
    where $J'\subset J_0$ is an $(m-l)$-plane. See Figure \ref{pic1}. We observe that for different $V_d$, $A_{\text{trans}}(m,V_d)$ are disjoint. We also observe that $\dim(A_{\text{trans}}(m,V_d))=(m+1)(d-m)$.

    We are ready to construct our set $\cV$. Choose a $\beta$-dimensional subset $Y\subset I_0\times I_2\times \dots\times I_{l}$. Let
    \[ \cV=\bigcup_{(y_0,y_2,\dots,y_{l})\in Y}  A_{\text{trans}}(m,V_d(y_0+J_0,y_1+J_0,\dots,y_{l}+J_0)). \]
    We have 
    \[ \dim(\cV)=\dim(Y)+\dim(A_{\text{trans}}(m,V_d))=(m+1)(d-m)+\beta. \]

    Note also that \[\bigcup_{V\in\cV}A(l,V)\subset \bigcup_{(y_0,y_1,\dots,y_{l})\in Y}  A_{\text{trans}}(l,V_d(y_0+J_0,y_1+J_0,\dots,y_{l}+J_0)). \]
    We have
    \[\dim(\bigcup_{V\in\cV}A(l,V))\le \dim(Y)+\dim(A(l,d))=(l+1)(d-l)+\beta.\]
    
\end{itemize}

\begin{figure}[ht]
\centering
\begin{minipage}[b]{0.85\linewidth}
\includegraphics[width=10cm]{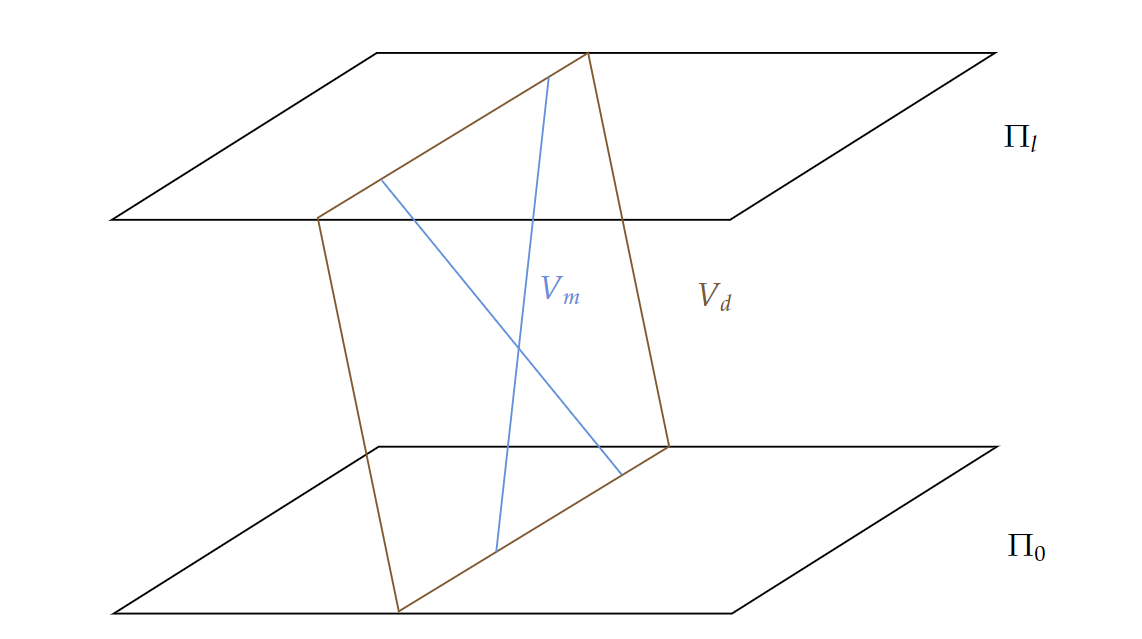}
\caption{}
\label{pic1}
\end{minipage}
\end{figure}

\medskip

\section{Preliminary}\label{sec3}

We introduce some notation that is needed for the later discussions. 

For a set $X$, we use $|X|$ to denote its measure, depending on the ambient space where $X$ lie. Here $X$ can be a subset in $\R^n$, $A(l,n)$ or $G(l,n)$. For any finite set $S$, we use $\#S$ to denote the number of elements in $S$.

Let $G(k,n)$ be the set of $k$-dimensional subspaces in $\R^n$. For every $k$-plane $V$, we can uniquely write it as
\[ V=\dir(V)+x_V, \]
where $\dir(V)\in G(k,n)$ and $x_V\in V^\perp$. $\dir(V)$ refers to the direction of $V$, as can be seen that $\dir(V)=\dir(V')\Leftrightarrow V\parallel V'$.

$A(k,n)$ is the set of $k$-planes $\R^n$. Since we can naturally view $A(k,n)$ as a subset of $G(k+1,n+1)$, the metric on $G(k+1,n+1)$ induces a metric on $A(k,n)$ for which we denote by $d(\cdot,\cdot)$. Actually, there are many equivalent metrics on $A(k,n)$ if we restrict to a local set in $A(k,n)$.

For $V_1, V_2\in A(k,n)$ with $x_{V_1},x_{V_2}\in B^n(0,1/2)$, we define $\rho(V_1,V_2)$ to be the smallest number $\rho$ such that $B^n(0,1)\cap V_1\subset N_{\rho}(V_2)$. We have the comparability of $d(\cdot,\cdot)$ and $\rho(\cdot,\cdot)$.

\begin{lemma}\label{comparablelem}
    There exists a constant $C>0$ (depending on $k,n$) such that for $V,V'\in A(k,n)$ with $x_{V},x_{V'}\in B^n(0,1/2)$,
    \[ C^{-1} d(V,V')\le \rho(V,V')\le C d(V,V').  \]
\end{lemma}

\begin{definition}
    For $V\in A(k,n)$ with $x_V\in B^n(0,1/2)$ and $0<r<1$, we define 
    \[V_r:=N_r(V)\cap B^n(0,1).  \]
\end{definition}
We see that $V_r$ is morally a slab of dimensions $\underbrace{r\times \dots\times r}_{n-k \textup{~times}}\times \underbrace{1\times \dots\times 1}_{k \tx{~times}}$. We usually call $V_r$ a $k$-dimensional $r$-slab. When $k$ is already clear, we simply call $V_r$ an $r$-slab.
If $W$ is a convex set such that $C^{-1} W\subset V_r\subset C W$, then we also call $W$ an $r$-slab. Here, the constant $C$ is a large number. For example, $C=100n$ will suffice.

\begin{definition}
    For two $r$-slab $V_r$ and $V'_r$. We say they are comparable if $C^{-1}V_r\subset V'_r\subset C V_r$.
\end{definition}

In this paper, we will also consider the balls and $\de$-neighborhoods in $A(k,n)$. Recall that we use $B_r(x)$ to denote the ball in $\R^n$ of radius $r$ centered at $x$. To distinguish the ambient space, we will use letter $Q$ to denote the balls in Grassmannians. For $V\in A(k,n)$, we use $Q_r(V)$ to denote the ball in $A(k,n)$ of radius $r$ centered at $V$. More precisely,
\[ Q_r(V):=\{V'\in A(k,n): d(V,V')\le r  \}. \]

For $X\subset \R^n$, we use $N_r(X)$ to denote the $r$-neighborhood of $X$ in $\R^n$.
For a subset $X\subset A(k,n)$, we use the fancy letter $\cN$ to denote the neighborhood in $A(k,n)$:
\[ \cN_r(X):=\{V\in A(k,n): d(V,X)\le r\}. \]
Here, $d(V,X)=\inf_{V'\in X}d(V,V')$.

\bigskip
We introduce the notation that is needed in the proof of the main theorem.
Let $\Pi_0=\R^{n-l}\times\{0\}^l$. Let $\Pi_j=\Pi_0+\vec\be_{n-l+j}$ for $j=1,\dots,l$. 
Define
\[ \Atr(l,n):=\{L\in A(l,n): L\textup{~is~transverse~to~} \Pi_0 \}. \]
We can identify  $\Atr(l,n)$ with $(\R^{n-l})^{l+1}$ through the map
\[ \Atr(l,n)\longrightarrow (\R^{n-l})^{l+1} \]
\[ L\mapsto (\Pi_0\cap L,\dots,\Pi_l\cap L). \]
Since we just need to work locally, we introduce the following local version.
\begin{definition}\label{def3.4}
For $L\in \Atr(l,n)$, let $x_j(L):=L\cap \Pi_j$  $(j=0,\dots,l)$.
Define
    \[ \Aloc(l,n):=\{L\in \Atr(l,n): x_j(L)\in [-1,1]^{n-l} \text{ for }j=0,\dots,l \}. \]
\end{definition}
From the definition, we can identify $\Aloc(l,n)$ with $[-1,1]^{(n-l)(l+1)}$.

\medskip

For $l<m\le n$, we can interpret an $m$-plane using the coordinate of $\Aloc(l,n)$.
We are only interested in those $m$-planes that interact with the $l$-planes in $\Aloc(l,n)$. More specifically, we give the following definition
\begin{definition}\label{def3.5}
For $V\in A(m,n)$ that is transverse to $\Pi_0$, we denote $v_j(V):=V\cap \Pi_j$ $(j=0,\dots,l)$ which are parallel $(m-l)$-planes.
Define
    \[\Aloc(m,n):=\{ V\in A(m,n): V\textup{~is~transverse~to~}\Pi_0,  v_j(V)\cap [-1,1]^{n-l}\neq \emptyset   \textup{~for~all~}j\}.\]
\end{definition}

 Here is a very intuitive way to think about their geometry. Each $V\in\Aloc(m,n)$ can be viewed as an $(l+1)$-tuple $(v_0,\dots,v_{l})$, where each $v_j=V\cap \Pi_j$ is an $(m-l)$-plane in $\Pi_j$ and all these $v_j$'s are parallel. Suppose we also have $L\in\Aloc(l,n)$ which is represented as $(x_0,\dots,x_{l})$ under the coordinate we just talked about. Then the incidence relationship $L\subset V$ is equivalent to 
\[ x_j\in v_j, \text{ for }j=0,\dots,l. \]

Next, we talk about the $\de$-discretization. For $L_1,L_2\in [-1,1]^{(n-l)(l+1)}=\Aloc(l,n)$, We say $L_1$ and $L_2$ are $\de$-separated if $\|L_1-L_2\|\le \de$, where $\|\ \|$ denotes the Euclidean distance in $[-1,1]^{(n-l)(l+1)}$. Since we are considering locally, the Euclidean distance is comparable to the distance induced by the Riemannian metric on $A(l,n)$. For $L\in [-1,1]^{(n-l)(l+1)}=\Aloc(l,n)$, we use $Q_\de(L)$ to denote the $\de$-ball centered at $L$ in $\Aloc(l,n)$, or we can view it as a $\de$-ball in $[-1,1]^{(n-l)(l+1)}$. For $V\in \Aloc(m,n)$, we use the bold-symbol $\bV$ to denote the set $\{L\in \Aloc(l,n): L\subset V\}$. The key point here is that $\bV$ is a subset of $\Aloc(l,n)$. We use $\cN_\de(\bV)$ to denote the $\de$-neighborhood of $\bV$ in $\Aloc(l,n)$. Using the coordinate $[-1,1]^{(n-l)(l+1)}$, we can view $\cN_\de(\bV)$ as $R_{0}\times \dots,\times R_l$, where $R_j\subset [-1,1]^{n-l}(\subset \Pi_j)$ are parallel rectangles of dimensions $\underbrace{\de\times \dots\times \de}_{(n-m)\textup{~times}}\times\underbrace{1\times \dots\times 1}_{(m-l)\textup{~times}}$.

Here are some facts about the measure which is easy to check: $|Q_\de(L)|\sim \de^{\dim(A(l,n))}= \de^{(n-l)(l+1)}$ for $L\in \Aloc(l,n)$; $|\cN_\de(\bV)|\sim \de^{\dim(A(l,n))-\dim(A(l,m))}= \de^{(l+1)(n-m)}$ for $V\in \Aloc(m,n)$.

\bigskip
\section{An \texorpdfstring{$L^p$}{} estimate for sum of \texorpdfstring{$k$}{}-geodesics}\label{sec4}

We first state the theorem that we will prove in this section.

\begin{theorem}\label{kakeyathm}
Let $0<l\le m\le  d< n$ be integers and $\beta\in[0,l+1]$. Let $\cV=\{V\}$ be a subset of $\Aloc(m,n)$. Suppose that for all balls $Q_r\subset \Aloc(m,n)$ of radius $r\in [\de,1]$, we have
\begin{equation}\label{condition}
    \#\{ V\in\cV: V\in Q_r \}\le (r/\de)^{(m+1)(d-m)+\beta}. 
\end{equation} 
Then
\begin{equation}\label{kakeyaineq}
    \| \sum_{V\in\cV} 1_{\cN_\de(\bV)}\|_{L^p(\Aloc(l,n))}\le C_{p,\e} \de^{-\frac{(m-l)(d-m)}{p'}-\e}\big(\sum_{V\in\cV}|\cN_\de(\bV)|\big)^{1/p},
\end{equation}
for $p$ in the range 
\begin{equation}\label{rangeofp}
    1\le p\le \min\left\{ \frac{d-k+\beta+1}{d-k+\beta}, \frac{d+\beta}{d+\beta-(d-m+2)^{-1}},d-m+2\right\}.
\end{equation}
Note that the right hand side is a number strictly bigger than $1$.
\end{theorem}

To prove Theorem \ref{kakeyathm}, we will use the broad-narrow method. To estimate the left hand side of \eqref{kakeyaineq}, we will decompose the integral into broad part and narrow part. The broad part will be handled by the Brascamp-Lieb inequality, and the narrow part will be handled by the induction. After Theorem \ref{kakeyathm} is proved, we will see that Theorem \ref{mainthm} is a corollary by H\"older's inequality.

\subsection{Broad-narrow reduction}

Choose $\L=\{L\}$ to be a maximal $\de$-separated subset of $\Aloc(l,n)$.
Then $\Aloc(l,n)$ is covered by finitely overlapping $\de$-balls $\{Q_\de(L): L\in\L\}$. 
We bound the left hand side of \eqref{kakeyaineq} as
\begin{equation}
    \begin{split}
        \| \sum_{V\in\cV}1_{\cN_\de(\bV)} \|_{L^p(\Aloc(l,n) )}^p&\le \sum_{L} \| \sum_{V\in\cV}1_{\cN_\de(\bV)} \|_{L^p(Q_\de(L))}^p \\ &\le\sum_{L} \#\{V\in\cV: \cN_{\de}(\bV)\cap Q_\de(L)\neq \emptyset  \}^p |Q_\de(L)|\\
        &\lesssim \sum_{L} \#\{V\in\cV: \cN_{\de}(\bV)\cap Q_\de(L)\neq \emptyset  \}^p \de^{-(l+1)(n-l)}.
    \end{split}
\end{equation}
In the last inequality, we use $|Q_\de(L)|\sim \de^{-\dim(A(l,n))}=\de^{-(l+1)(n-l)}$.

For $L\in\L$, let
\[ \cV(L):=\{V\in\cV: \cN_{\de}(\bV)\cap Q_\de(L)\neq \emptyset\}.\]
Then we have
\begin{equation}\label{recalling}
    \| \sum_{V\in\cV}1_{\cN_\de(\bV)} \|_{L^p(\Aloc(l,n) )}^p\lesssim\sum_{L} \#\cV(L)^p \de^{-(l+1)(n-l)}.
\end{equation}
Our goal is to find an upper bound of $\#\cV(L)$.

Since the $m$-planes in $\cV(L)$ contains $L$ up to uncertainty $\de$, we may just assume $L\subset V$ for $V\in \cV(L)$. This is harmless, since later we will pass to their $\de$-neighborhoods. Next, we discuss how to view $\cV(L)$ as a subset of $G(m-l,n-l)$.

We introduce a new notation.
\begin{definition}\label{defbw}
    For $V\in \Aloc(m,n)$, we define $\vec\bw (V)$ to be the $(m-l)$-dimensional plane in $G(m-l,n-l)$ that is parallel to $\Pi_0\cap V$.
\end{definition}

Choose $\D\subset G(m-l,n-l)$ to be a maximal $\de$-separated subset. For any $L\in \Aloc(l,n)$ and $V\in \Aloc(m,n)$ such that $L\subset V$, we see that $v_j(V)=x_j(L)+\vec\bw(V)$ ($j=0,\dots,j$) (recalling Definition \ref{def3.4}, \ref{def3.5}). Therefore, by restricting to $\Pi_0$, we can view $\cV(L)$ as a ``bush" of $(m-l)$-planes centered at $x_0(L)$. (See Figure \ref{pic2}.)
Since $\cV(L)$ is $\de$-separated, we can view $\cV(L)$ as a subset of $\D(\subset G(m-l,n-l))$. 

\begin{figure}[ht]
\centering
\begin{minipage}[b]{0.85\linewidth}
\includegraphics[width=9cm]{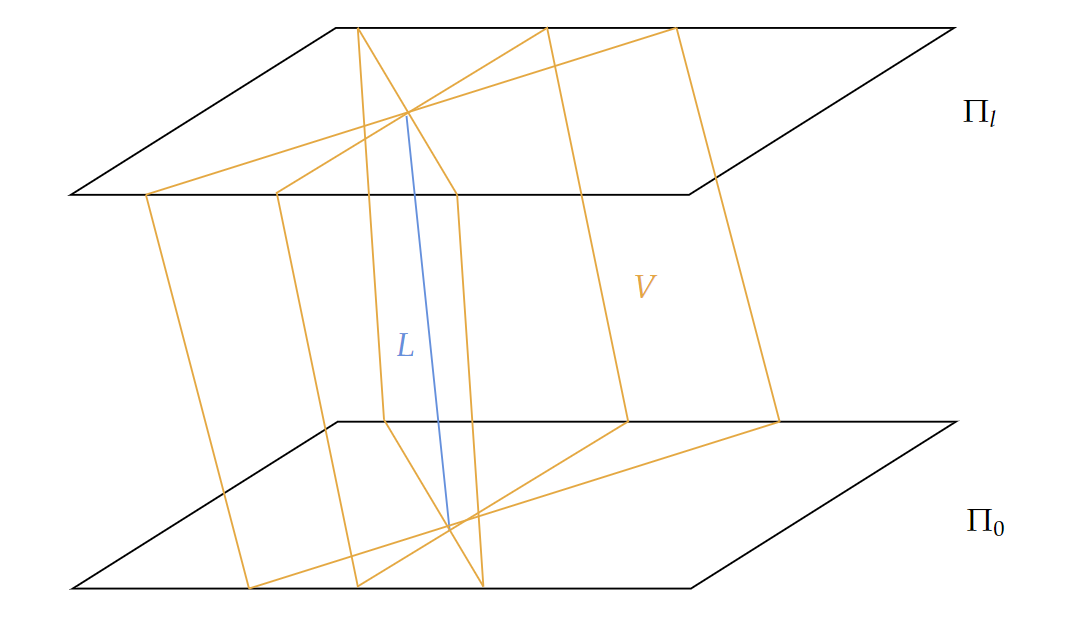}
\caption{}
\label{pic2}
\end{minipage}
\end{figure}

\bigskip

Now, we fix an $L\in \L$, and do broad-narrow argument for $\cV(L)$.
Fix a large number $K\gg1$ which is to be determined later. We cover $G(m-l,n-l)$ by $\cQ_{K^{-n}}=\{Q_{K^{-n}}\}$ which are balls of radius $K^{-(n-l)}$ with bounded overlaps: 
\[ G(m-l,n-l)=\bigcup Q_{K^{-n}}. \]
Note that $\#\cQ_{K^{-n}}\lesssim (K^{n-l})^{\dim G(m-l,n-l)}\le K^{n^3}$. We define the significant balls
\[ \cQ'(L)=\{ Q_{K^{-n}}: \#(Q_{K^{-n}}\cap \cV(L))\ge \frac{1}{K^{n^4}}\#\cV(L) \}. \]
We have 
\[ \#\cV(L)\le 2 \sum_{Q_{K^{-n}}\in \cQ'(L)}\#(Q_{K^{-n}}\cap \cV(L)). \]
Next by pigeonhole principle, there is a subset $\cQ(L)\subset \cQ'(L)$, such that the balls in $\cQ(L)$ are $100K^{-n}$-separated, $\#(Q_{K^{-n}}\cap \cV(L))$ are all comparable for $Q_{K^{-n}}\in\cQ$, and
\[ \#\cV(L)\lesssim \log K \sum_{Q_{K^{-n}}\in\cQ(L)}\#(Q_{K^{-n}}\cap \cV(L)). \]

Next, we determine whether $L$ is broad or  narrow. 

\begin{definition}
Fix an $L\in\L$, and let $\cQ(L)$ be given above.
    We say $L$ is \textbf{narrow} if there exists a $(d-l)$-plane $\Pi\in G(d-l,n-l)$ so that half of the balls in $\cQ(L)$ are contained in $\cN_{C_1K^{-1}}(G(m-l,\Pi))$. Here, we view $G(m-l,\Pi)$ as a subset of $G(m-l,n-l)$ and $\cN_{C_1K^{-1}}(G(m-l,\Pi))$ is the $C_1K^{-1}$-neighborhood of $G(m-l,\Pi)$ in $G(m-l,n-l)$. $C_1$ is a large constant to be determined later. We say $L$ is \textbf{broad} otherwise.
\end{definition}

The next lemma is the property for a broad $L$.

\begin{lemma}\label{broadpro}
    Suppose $L$ is broad. Then there exists 
\[\cU_{m-l},\dots,\cU_{d-l+1}\in \cQ(L),\]
so that there does not exist a $\Pi\in G(d-l,n-l)$ such that all of these $\cU_j$  intersect $\cN_{100K^{-(n-l)}}(G(m-l,\Pi))$. In other words, for any $\Pi\in G(d-l,n-l)$, there exists a $\cU_j$ that does not intersect $\cN_{100K^{-(n-l)}}(G(m-l,\Pi))$. We remark that the subscript of $\cU_j$ ranges from $m-l$ to $d-l+1$ but not from $1$ to $d-m+2$, just to make the discussion below more convenient.
\end{lemma}

\begin{proof}

We want to show that if $L$ is not narrow, then the property in Lemma \ref{broadpro} is satisfied. 
We let $\cG$ be the set of centers of the balls in $\cQ(L)$, i.e., 
\[ \cQ(L)=\{Q_{K^{-n}}(U):U\in\cG\}. \]
We just need to show that there exists $U_{m-l},\dots,U_{d-l+1}\in\cG$ so that there does not exists a $\Pi\in G(d-l,n-l)$ such that all these $U_j$ are contained in $\cN_{200K^{-(n-l)}}(G(m-l,\Pi))$.

If $U_{m-l},\dots,U_{d-l+1}\in \cN_{200K^{-(n-l)}}(G(m-l,\Pi))$, then by Lemma \ref{comparablelem}, 
\begin{equation}\label{nothappen}
    U_{m-l},\dots,U_{d-l+1}\subset N_{CK^{-n}}(\Pi)
\end{equation}
for some large constant $C$ depending on $k,n$. Here, \eqref{nothappen} happens in the ambient space $\R^{n-l}$.

To show the existence of $U_{m-l},\dots,U_{d-l+1}$ that do not satisfy \eqref{nothappen}, we will inductively find $U_{m-l},\dots,U_{d-l+1}$ and vectors $v_1,\dots,v_{d-l+1}$ in $\bigcup_{j=m-l}^{d-l+1}U_j$ such that $v_1\dots,v_{d-l+1}$ cannot all lie in $N_{CK^{-n}}(\Pi)$.

We observe the following fact: If the unit vectors $v_1,\dots,v_{d-l+1}$ lie in $CK^{-n}$-neighborhood of a $(d-l)$-dimensional plane $\Pi$, then the volume of the parallelepiped spanned by $\{v_1,\dots,v_{d-l+1}\}$ has the upper bound: 
\[\tx{Vol}_{d-l+1}(v_1,\dots,v_{d-l+1})\lesssim  K^{-n}.\] We give the proof. Let $\pi=\tx{span}\{v_1,\dots,v_{d-l+1}\}$, then the parallelepiped spanned by $\{v_1,\dots,v_{d-l+1}\}$ is contained in $N_{CK^{-n}}(\Pi)\cap \pi\cap B^{n-l}(0,1)$. We get 
\begin{equation}
    \tx{Vol}_{d-l+1}(v_1,\dots,v_{d-l+1})\le \tx{Vol}_{d-l+1}\bigg( N_{CK^{-n}}(\Pi)\cap B^{n-l}(0,1)\cap \pi \bigg).
\end{equation}
Noting that $N_{CK^{-n}}(\Pi)\cap B^{n-l}(0,1)$ is morally a rectangle of dimensions \[\underbrace{1\times \dots\times 1}_{d-l \tx{~times} }\times \underbrace{CK^{-n}\times \dots \times CK^{-n}}_{(n-d)\tx{~times} }.\]
Its intersection with a $(d-l+1)$-dimensional plane $\pi$ has volume $\lesssim K^{-n}$.

Therefore, we just need to find $U_{m-l},\dots,U_{d-l+1}\in\cG$ and vectors $v_1,\dots,v_{d-l+1}$ that lie in $\bigcup_{j=m-l}^{d-l+1} U_j$ such that
\[ \tx{Vol}_{d-l+1}(v_1,\dots,v_{d-l+1})\ge C' K^{-n}, \]
for some large constant $C'$. This will imply that $U_{m-l},\dots,U_{d-l+1}$ cannot all lie in $\cN_{200K^{-n}}(G(m-l,\Pi))$.

This is done in the following way. 
We first choose a $U_{m-l}\in\cG$ and choose $v_1,\dots,v_{m-l}$ to be the orthonormal basis of $U_{m-l}$.
Next, we will inductively construct $U_{j}$ for $m-l<j\le d-l+1$ and a unit vector $v_j\in U_j$ so that the $j$-dimensional parallelepided spanned by $v_1,\dots,v_j$ has volume bigger than $\wt C^jK^{-j}$ for some large constant $\wt C$. The base case is when $j=m-l$. By our choice, we have
\[ \textup{Vol}_k(v_1,\dots,v_{m-l})=1\gtrsim  \wt C^{m-l}K^{-(m-l)}. \]
Suppose we have constructed $U_{m-l},\dots,U_{j-1}$ and $v_1,\dots,v_{j-1}$ with 
\[\textup{Vol}_{j-1}(v_1,\dots,v_{j-1})\ge \wt C^{j-1}K^{-(j-1)}.\]
Denote the $(j-1)$-dimensional plane $\textup{span}\{v_1,\dots,v_{j-1}\}$ by $\pi_{j-1} (\subset \R^{n-l})$. For any $U\in\cG$, consider its relative position with $\pi_{j-1}$. It suffices to find a unit vector $v\in U$ that is not contained in the $\wt C K^{-1}$-neighborhood of $\pi_{j-1}$. The key observation is that: If every unit vector $v\in U$ is contained in the $\wt CK^{-1}$-neighborhood of $\pi_{j-1}$, then $U$, as an element in $G(m-l,n-l)$, is contained in $\cN_{C_1 K^{-1}}(G(m-l,\pi_{j-1}))$. (Here, $C_1$ is the constant in the definition of narrowness, which is determined after $\wt C$.) This contradicts the assumption that $L$ is not narrow.
\end{proof}

\bigskip

Now, we can estimate $\#\cV(L)$. Recall $\cQ_{K^{-n}}=\{Q_{K^{-n}}\}$ is a partition of $G(m-l,n-l)$ by $K^{-n}$-balls. Let $\cQ_{K^{-1}}=\{Q_{K^{-1}}\}$ be a partition of $G(m-l,n-l)$ by $K^{-1}$-balls.  We can also assume they are nested in the sense that each $Q_{K^{-n}}$ is contained in some $Q_{K^{-1}}$.

Recall that $\Aloc(l,n)=[-1,1]^{(n-l)(l+1)}=\big([-1,1]^{(n-l)}\big)^{(l+1)}$. For $V\in \Aloc(m,n)$, recall in Definition \ref{defbw} that $\vec\bw(V)\in G(m-l,n-l)$ is the $(m-l)$ dimensional subspace parallel to $V\cap \Pi_0$.

Define the following transverse $(d-m+2)$-tuples
\begin{align}
    \textup{Trans}^{d-m+2}:=\{ (\cU_{m-l},\dots,\cU_{d-l+1})\in (\cQ_{K^{-n}})^{d-m+2}:\cU_{m-l},\dots,\cU_{d-l+1} \textup{~do~not~simultaneously}\\
   \nonumber \textup{~interset~}\cN_{100K^{-n}}(G(m-l,\Pi)) \textup{~for~any~} \Pi\in G(d-l,n-l)  \}.
\end{align}
We have the following estimate for $\#\cV(L)$.

\begin{proposition}\label{bnest} 
For $L\in \L$, we have
\begin{align}
    \#\cV(L)^p\lesssim K^{(m-l)(d-m)(p-1)}(\log K)^p\sum_{Q_{K^{-1}}\in\cQ_{K^{-1}}}\#(Q_{K^{-1}}\cap\cV(L))^p\\
    \label{huge}+K^{O(1)} \sum_{(\cU_{m-l},\dots,\cU_{d-l+1})\in\textup{Trans}^{d-m+2}}\prod_{j=m-l}^{d-l+1}(\#\cU_j\cap\cV(L))^{\frac{p}{d-m+2}}.
\end{align}    
\end{proposition}

\begin{proof}
If $L\in\L$ is narrow, then there exists $\Pi\in G(d-l,n-l)$ such that
\[ \#\cV(L)\lesssim \log K \sum_{Q_{K^{-1}}\subset \cN_{C_1K^{-1}}(G(m-l,\Pi))}\#(Q_{K^{-1}}\cap \cV(L)). \]
Since $G(m-l,\Pi)$ is a $(m-l)(d-m)$-dimensional submanifold of $G(m-l,n-m)$, we have 
\[\#\{Q_{K^{-1}}\subset \cN_{C_1K^{-1}}(G(m-l,\Pi))\}\lesssim K^{(m-l)(d-m)}. \]
Therefore, we have

\begin{equation}
    \#\cV(L)^p\lesssim K^{(m-l)(d-m)(p-1)} (\log K)^p \sum_{Q_{K^{-1}}\in \cQ_{K^{-1}}}\#(Q_{K^{-1}}\cap \cV(L))^p.
\end{equation}

We discuss the bound when $L$ is broad. In this case, we have
\begin{equation}
\begin{split}
     \#\cV(L)\lesssim K^{O(n^{10})} \max_{(\cU_{m-l},\dots,\cU_{d-l+1})\in\textup{Trans}^{d-m+2}}\prod_{j=m-l}^{d-l+1}(\#\cU_j\cap\cV(L))^{\frac{1}{d-m+2}}\\
     \le K^{O(n^{10})} \sum_{(\cU_{m-l},\dots,\cU_{d-l+1})\in\textup{Trans}^{d-m+2}}\prod_{j=m-l}^{d-l+1}(\#\cU_j\cap\cV(L))^{\frac{1}{d-m+2}}
\end{split}
    \end{equation}
Combining the broad case and narrow case, we obtain the estimate we want.
\end{proof}

\bigskip

Recalling \eqref{recalling}, we have
\[ \| \sum_{V\in\cV}1_{\cN_\de(\bV)} \|_{L^p(\Aloc(l,n) )}^p\lesssim\sum_{L} \#\cV(L)^p \de^{-(l+1)(n-l)}. \]
Using Proposition \ref{bnest} and noting that each $\cN_\de(\bV)$ is roughly constant on any ball of radius $\de$ in $\Aloc(l,n)$, we obtain

\begin{equation}\label{1term}
\begin{split}
&\|\sum_{V\in\cV}1_{\cN_\de(\bV)}\|_{L^p(\Aloc(l,n))}^p\\
&\lesssim K^{(m-l)(d-m)(p-1)}(\log K)^p\sum_{Q_{K^{-1}}\in\cQ_{K^{-1}}}\| \sum_{V\in\cV, \vec\bw(V)\in Q_{K^{-1}}}1_{\cN_\de(\bV)} \|_{L^p(\Aloc(l,n))}^p\\
    &+C_K\sum_{(\cU_{m-l},\dots,\cU_{d-l+1})\in\textup{Trans}^{d-m+2}}\int_{\Aloc(l,n)}\prod_{j=m-l}^{d-l+1}(\sum_{V\in\cV, \vec\bw(V)\in \cU_j}1_{\cN_\de(\bV)})^{\frac{p}{d-m+2}}\\
    &=:\textup{I}+\textup{II}.
\end{split}\end{equation}

\bigskip

Next, we will estimate I and II separately. We first look at II.

\subsection{Estimate of \textup{II}}

\subsubsection{Brascamp-Lieb inequality}\hfill

The version of Brascamp-Lieb inequality we are going to use is due to Maldague. See Theorem 2 in \cite{maldague2022regularized}.

\begin{theorem}[Maldague]\label{Domthm}
Consider in $\R^N$. Fix $1\le k\le N$, $J\in \N$.
    Let $W_j\subset G(N-k,N)$ for $j=1,\dots,J$. Fix $p\in[1,J]$. Define
    \begin{equation}\label{bl0}
        \BL(\{W_j\}_{j=1}^J,p):=\sup_{U\leq \R^N }(\dim U-\frac{p}{J}\sum_{j=1}^J\dim\pi_{W_j}(U)). 
    \end{equation} 
Here, $U\le \R^N$ means $U$ is a subspace of $\R^N$, and $\pi_{W_j}:\R^N\rightarrow W_j$ is the orthogonal projection. 

There exists $\nu>0$ depending on $\{W_j\}_{j=1}^J$, so that the following is true. For any $\cV_{j}=\{V_{j}\}\subset A(k,N)\ (j=1,\dots,J)$ being sets of $k$-planes, such that each $V_{j}\in \cV_{j}$ is orthogonal to some $W\in Q_\nu(W_j)$, we have
\begin{equation}\label{mblineq}
    \int \prod_{j=1}^J\bigg(\sum_{V_{j}\in\cV_{j}}1_{V_{j,\de}} \bigg)^{\frac{p}{J}}\lesssim \de^{N-\e}\de^{-\BL(\{W_j\}_{j=1}^J,p)}\prod_{j=1}^J\bigg(\#\cV_{j}\bigg)^{\frac{p}{J}}.
\end{equation}
\end{theorem}

\begin{remark}
    {\rm  Later, we will let $N=(n-l)(l+1)$ and $k=(m-l)(l+1)$ for our application.
    }
\end{remark}

By a compactness argument, we deduce the following result.

\begin{theorem}\label{BLthm}
Consider in $\R^N$.
    Let $\cW_j\subset G(N-k,N)$ be a compact set for $j=1,\dots,J$. Fix $p\in[1,J]$. 
Define
\begin{equation}\label{bl}
\BL(\{\cW_j\}_{j=1}^J,p):=\sup_{W_1\in\cW_1,\dots,W_J\in\cW_J }\BL(\{W_j\}_{j=1}^J,p). 
\end{equation} 
Then for any $\cV_{j}=\{V_{j}\}\subset A(k,n)\ (j=1,\dots,J)$ being sets of $k$-planes, such that each $V_{j}\in \cV_{j}$ is orthogonal to some $W\in\cW_j$, we have
\begin{equation}\label{BL}
    \int\prod_{j=1}^J\bigg(\sum_{V_{j}\in\cV_{j}}1_{V_{j,\de}} \bigg)^{\frac{p}{J}}\le C(\{\cW_j\},p,\e) \cdot \de^{N-\e}\de^{-\BL(\{\cW_j\}_{j=1}^J,p)}\prod_{j=1}^J\bigg(\#\cV_{j}\bigg)^{\frac{p}{J}}.
\end{equation}
\end{theorem}

\begin{proof}
    We briefly explain how to use compactness argument to deduce Theorem \ref{BLthm} from Theorem \ref{Domthm}. For any $\vec W=(W_1,\dots,W_J)\in \cW_1\times \dots\times \cW_J$, we use Theorem \ref{Domthm} to find a small number $\nu>0$, so that \eqref{BL} is true when all the slabs are in $\nu$-neighborhood of $W_j^\perp$. In other words, there exists a $O(\nu)$-neighborhood of $(W_1,\dots,W_J)$, for which we call $B(\vec W)$, so that \eqref{BL} is true when all the slabs are in $B(\vec W)$. Then we use the compactness of $\cW_1\times \dots\times \cW_J$ to find a finite covering using $B(\vec W)$. Using this finite covering and splitting the LHS of \eqref{BL} by triangle inequality, the proof of \eqref{BL} is not hard.
\end{proof}

\bigskip

\subsubsection{Geometry of \texorpdfstring{$\Aloc(l,n)$}{}}\hfill

Recall that we can view $\Aloc(l,n)$ as $[-1,1]^{(n-l)(l+1)}$ which is a cube of bounded length centered at the origin in $\R^{(n-l)(l+1)}$. We will interpret $\Aloc(m,n)$ in the following way. Actually $\Aloc(m,n)$ can be viewed as a subset of $A((m-l)(l+1),(n-l)(l+1))$. For simplicity, we denote $k=(m-l)(l+1), N=(n-l)(l+1)$. For $V\in\Aloc(m,n)$, we see that $v_j(V)=V\cap \Pi_j$ is an $(m-l)$-plane for $j=0,\dots,l$. Therefore, we obtain a $k$-plane $v_0(V)\times \dots \times v_l(V)\subset \R^N$.
We summarize our embeddings:
\begin{equation}
    \begin{split}
        \Aloc(l,n)&\longrightarrow \R^N,\\
        \Aloc(m,n)&\longrightarrow A(k,N).
    \end{split}
\end{equation}
We adopt the following notation.
\begin{definition}\label{deftilde}
    For $V\in \Aloc(m,n)$, we use $\wt V\in A(k,N)$ to denote its image in the new space given by the correspondence.
\end{definition}

We remark that while $\Aloc(l,n)$ and $\R^N$ have the same dimension, $\Aloc(m,n)$ is a lower dimensional submanifold in $A(k,N)$.

The incidence relation is preserved. For $L\in\Aloc(l,n)$ and $V\in \Aloc(m,n)$, then $L\subset V$, if and only if $\wt L\in \wt V$. 

The notion of parallelity is also preserved. Suppose $V\in\Aloc(m,n)$. Let $W=\vec\bw (V)$. Let $W^\perp$ be the orthogonal complement of $W$ in $\R^{n-l}$. Then, we see that
$\wt V=v_0(V)\times \dots\times v_l(V)$ has its orthogonal complement $\underbrace{W^\perp\times \dots\times W^\perp}_{l+1 \text{ times}}$ in $\R^N$. Therefore, for $V_1,V_2\in\Aloc(m,n)$, we have $\vec\bw(V_1)=\vec\bw(V_2)$ if and only if $\wt V_1\parallel \wt V_2$. Because of this observation, we can view $G(m-l,n-l)$ as a subset of $G(k,N)$, so that any $\vec\bw\in G(m-l,n-l)$ corresponds to $\underbrace{\vbw\times \dots\times \vbw}_{l+1 \text{ times}}$. The set of balls $\cQ_{K^{-1}}=\{Q_{K^{-1}}\}$ correspond to $\wt\cQ_{K^{-1}}=\{\wt Q_{K^{-1}}\}$ which form a partition of the $K^{-1}$-neighborhood of $G(m-l,n-l)$ in $G(k,N)$.

We define the transverse tuples under the setting in $\R^N$.
\begin{align}
    \wt{\text{Trans}}^{d-m+2}:
    =\{ (\wt\cU_{m-l},\dots,\wt\cU_{d-l+1}): (\cU_{m-l},\dots,\cU_{d-l+1})\in \textup{Trans}^{d-m+2}\}.
\end{align}
We have the transverse property.

\begin{lemma}\label{transvers}
    Let $(\wt\cU_{m-l},\dots,\wt\cU_{d-l+1})\in \wt{\text{Trans}}^{d-m+2}$ and $\wt U_j\in \wt \cU_j$. Then any $\Pi'\in G((d-l)(l+1)+l,N)$ cannot contain all the $\wt U_j$. 
\end{lemma}
\begin{proof}
    By definition, we have $(\cU_{m-l},\dots,\cU_{d-l+1})\in \text{Trans}^{d-m+2}$, $U_j\in \cU_j$ and
    \[ \wt U_j=U_j\times \dots\times U_j. \]

    Suppose by contradiction that there exists $\Pi'\in G((d-l)(l+1)+l,N)$ that contains $\wt U_j$ for all $j$. For $t=0,1,\dots,l$, consider $\Pi'\cap \big(\{0\}^{(n-l)t}\times \R^{n-l}  \times \{0\}^{(n-k)(l-t)}\big)$. By pigeonhole principle, there exists $t$ such that 
    \[ \dim\bigg(\Pi'\cap \big(\{0\}^{(n-l)t}\times \R^{n-l}  \times \{0\}^{(n-k)(l-t)}\big)\bigg)\ge \lfloor \frac{(d-l)(l+1)+l}{l+1}\rfloor=d-l. \]
    Denote $\Pi=\Pi'\cap \big(\{0\}^{(n-l)t}\times \R^{n-l}  \times \{0\}^{(n-k)(l-t)}\big)$.
    Then $\Pi\in G(d-l,n-l)$, and $\Pi$ contains $U_j$ for all $j$. This contradicts the definition of $\text{Trans}^{d-m+2}$. 
\end{proof}

Then, passing from $\Aloc(l,m)$ to $\R^N$, we have
\begin{equation}\label{2term}
    \text{II}\sim C_K\sum_{(\wt\cU_{m-l},\dots,\wt\cU_{d-l+1})\in\wt{\textup{Trans}}^{d-m+2}}\int_{\R^N}\prod_{j=m-l}^{d-l+1}(\sum_{V\in\cV, \dir(\wt V)\in \wt\cU_j}1_{\wt V_\de})^{\frac{p}{d-m+2}}
\end{equation}

\bigskip

\subsubsection{The estimate of \textup{II}}

\begin{proposition}
 \begin{equation}\label{estII}
    \tx{II}\le C_{K,\e}\de^{-\e^2-(m-l)(d-m)(p-1)}\sum_{V\in\cV}|\cN_\de(\bV)|.
\end{equation}   
\end{proposition}
\begin{proof}
Denote $J=d-m+2$. For each $\wt\cU_j$ ($m-l\le j\le d-l+1$), let $\cW_j=\{W^\perp:W\in \wt\cU_j\}$. We see that $\cW_j\subset G(N-k,N)$. For simplicity, we denote the Brascamp-Lieb  constant (see \eqref{bl} and \eqref{bl0}) by
\begin{equation}\label{recallBL}
    \BL:=\BL(\{\cW_j\}_{j=1}^J,p). 
\end{equation} 
Noting from \eqref{rangeofp} that $p\le d-m+2=J$, by Theorem \ref{BLthm}, we have
\begin{equation}
    \int_{\R^N} \prod_{j=m-l}^{d-l+1}(\sum_{V\in\cV, \dir(\wt V)\in \wt\cU_j}1_{\wt V_\de})^{\frac{p}{d-m+2}}\lesssim \de^{N-\e^2}\de^{-\BL}\prod_{j=1}^J\bigg(\#\wt\cU_{j}\cap \cV\bigg)^{\frac{p}{J}}\le\de^{N-\e^2-\BL}(\#\cV)^p.
\end{equation}
Here, the implicit constant depends on $\e$ and $\{\cW_j\}$. $\{\cW_j\}$ further depends on the covering $\cQ_{K^{-n}}$. Since this covering is at the scale $K^{-n}$, we can simply denote this implicit constant by $C_{K,\e}$. 
Also, the reason we choose $\de^{-\e^2}$ instead of $\de^{-\e}$ is to close the induction later.

Therefore,
\[ \tx{II}\le C_{K,\e} \de^{N-\e^2-\BL}(\#\cV)^p.\]
Here, we absorb $C_K$ from \eqref{2term} into $C_{K,\e}$.

It remains to prove that 
\begin{equation}\label{clm}
    \de^{N-\BL}(\#\cV)^p\lesssim \de^{-(m-l)(d-m)(p-1)}\sum_{V\in\cV}|\cN_\de(\bV)|. 
\end{equation}
To prove this inequality, we first note that $|\cN_\de(\bV)|\sim |\wt V_\de|\sim \de^{N-k}$ and by \eqref{condition} that $\#\cV\le \de^{-(m+1)(d-m)+\beta}$.
Therefore, \eqref{clm} is equivalent to
\[ \de^{N-\BL}\de^{-\big((m+1)(d-m)+\beta\big)(p-1)}\lesssim \de^{-(m-l)(d-m)(p-1)}\de^{N-k}. \]
Recall that $N=(l+1)(n-l), k=(l+1)(m-l)$.
It suffices to prove
\begin{equation}\label{boundBL}
    \BL\le (l+1)(d-l)+\beta-\big((l+1)(d-m)+\beta\big)p. 
\end{equation} 
Recall the definition of $\BL$ in \eqref{recallBL}. We may assume
\[ \BL=\sup_{U\le \R^N}\bigg(\dim U-\frac{p}{J}\sum_{j=k}^{d+1}\dim \pi_{W_j}(U)\bigg), \]
where $W_j\in \cW_j$.
\begin{lemma}\label{easylem}
    Let $W,U\le \R^N$ be two subspaces. Then,
     \[\dim\pi_{W}(U)\ge \dim W+\dim U-N.\] 
     The equality holds if and only if $U\supset W^\perp$. If $U\not\supset W^\perp$, then \[\dim\pi_{W}(U)\ge \dim W+\dim U-N+1.\]
\end{lemma}
\begin{proof}
    We just look at the linear map $\pi_W|_U: U\rightarrow \pi_W(U)$. We have $\dim(\pi_W(U))=\dim(U)-\dim\textup{Ker}(\pi_W|_U)$. We note that $\dim\textup{Ker}(\pi_W|_U)\le \dim\textup{Ker}(\pi_W)=N-\dim W$, and $\dim\textup{Ker}(\pi_W|_U)= \dim\textup{Ker}(\pi_W)$ if and only if $U\supset W^\perp$.
\end{proof}

We return to the proof of \eqref{boundBL}. Recall that $J=d-m+2$.
We denote $u=\dim U$ and consider three cases.
\medskip

\begin{itemize}
    \item $u\ge (l+1)(d-l)+\beta$: $\dim U-\frac{p}{J}\sum_{j=m-l}^{d-l+1}\dim \pi_{W_j}(U)\le u-p(N-k+u-N)=u-p(u-k)\le (l+1)(d-l)+\beta-\big((l+1)(d-m)+\beta\big)p$, since $p>1$.
    \medskip
    \item $ u< (l+1)(d-l)+\beta$: In this case, $u\le (l+1)(d-l)+l$. By Lemma \ref{transvers}, $U$ cannot contain all the $W_j^\perp$. By Lemma \ref{easylem}, we have \[\sum_{j=m-l}^{d-l+1}\dim\pi_{W_j}(U)\ge J(u-k)+1 .\]
    Therefore,
    $\dim U-\frac{p}{J}\sum_{j=m-l}^{d-l+1}\dim \pi_{W_j}(U)\le u-p(u-k)-\frac{p}{J}\le (l+1)(d-l)+\beta-\big((l+1)(d-m)+\beta\big)p$. The last inequality is equivalent to $\big((l+1)(d-l)+\beta-u-J^{-1}\big)p\le (l+1)(d-l)+\beta-u$. Since the RHS is non-negative, the inequality is always true if $(l+1)(d-l)+\beta-u-J^{-1}\le 0$. Or when $(l+1)(d-l)+\beta-u-J^{-1}>0$,
    we just need $p\le \frac{(l+1)(d-l)+\beta-u}{(l+1)(d-l)+\beta-u-J^{-1}}$.
\end{itemize}
We have finished the proof of claim \eqref{clm}. We obtain that 
\begin{equation}
    \tx{II}\le C_{K,\e}\de^{-\e^2-(m-l)(d-m)(p-1)}\sum_{V\in\cV}|\cN_\de(\bV)|.
\end{equation}
\end{proof}

\bigskip

\subsection{Estimate of \textup{I} and induction on scales}\hfill

We turn to the term $\tx{I}$, which is the first term on the right hand side of \eqref{1term}. We will first do a rescaling to rewrite I in the form of the LHS of \eqref{kakeyaineq}, so that we can do induction on scales. Also, we will check the corresponding spacing condition \eqref{condition} is satisfied.

For a fixed $Q_{K^{-1}}\in \cQ_{K^{-1}}$, we choose $W^\circ\in G(m-l,n-l)$ to be the center of $Q_{K^{-1}}$. Consider the following way of partitioning $\Aloc(l,n)$ into about $K^{(n-m)(l+1)}$ many $K^{-1}$-thickened slabs $\cN_{K^{-1}}(\bM)$, where $M\in\Aloc(m,n)$ and $\bM=\{L\in\Aloc(l,n): L\subset M\}$.. Recall that $\Aloc(l,n)=\underbrace{[-1,1]^{n-l}\times \dots\times [-1,1]^{n-l}}_{l+1 \textup{~times}}$. For $j=0,1,\dots,l$, the $j$-th copy of $[-1,1]^{n-l}$ lies in $\Pi_j$. We choose $(m-l)$-dimensional $K^{-1}$-slabs in $\Pi_j$, so that their cores are parallel to $W^\circ$, and they form a covering of $[-1,1]^{n-l}$ with bounded overlap. We denote their cores by $W^j_{1},\dots,W^j_{\mu}$, where $\mu\sim K^{n-m}$. Hence,
\[ [-1,1]^{n-l}=\bigcup_{\alpha=1}^\mu N_{K^{-1}}(W_\alpha^j). \]
Define $\M=\M_{Q_{K^{-1}}}=\{M\}$, where each $M\in \Aloc(m,n)$ has the form
\[ \wt M=W_{\alpha_0}^0\times W_{\alpha_1}^1\times \dots\times W_{\alpha_l}^l. \]
We see that
\[ \Aloc(l,n)=\bigcup_{M\in\M}\cN_{K^{-1}}(\bM) \]
is a covering with bounded overlap. Also, $\#\M\sim K^{(n-m)(l+1)}$.

By the identification $\Aloc(l,n)=[-1,1]^N$, we can identify $\cN_{K^{-1}}(\bM)$ with $N_{K^{-1}}(\wt M)\cap [-1,1]^N$. Interested reader can further check that $\{\wt M: M\in\M\}$ are parallel $k$-dimensional planes in $\R^N$, though we do not need this fact.
We note that for any $V\in \Aloc(m,n)$ with $\vec\bw(V)\in Q_{K^{-1}}$, there exists $M\in \M_{Q_{K^{-1}}}$ such that $\cN_\de(\bV)\subset \cN_{K^{-1}}(\bM)$. Actually, when viewing them in $\R^N$, $N_{K^{-1}}(\wt V)$ and $N_{K^{-1}}(\wt M)$ are comparable and $N_{\de}(\wt V)\subset N_{K^{-1}}(\wt M)$. We assign $V$ to such an $M$ and denote by $V\prec M$. 



To estimate $\| \sum_{V\in\cV, \vec\bw(V)\in Q_{K^{-1}}}1_{\cN_\de(\bV)} \|_{L^p(\Aloc(l,n))}^p$, we partition the integration domain to get
\begin{equation}\label{ineq1}
    \| \sum_{V\in\cV, \vec\bw(V)\in Q_{K^{-1}}}1_{\cN_\de(\bV)} \|_{L^p(\Aloc(l,n))}^p\sim\sum_{M\in\M}\| \sum_{V\in\cV, \vec\bw(V)\in Q_{K^{-1}}, V\prec M}1_{\cN_\de(\bV)} \|_{L^p(\cN_{K^{-1}}(\bM))}^p.
\end{equation}

Fix an $M\in\M$, we will do the rescaling for each $\cN_{K^{-1}}(\bM)$ so that under the rescaling, $\cN_\de(\bV)$ becomes $\cN_{\de K}(\bV')$ and $\cN_{K^{-1}}(\bM)$ becomes $\Aloc(l,n)$. Suppose 
\[ \wt M=W_{\alpha_0}^0\times W_{\alpha_1}^1\times \dots\times W_{\alpha_l}^l. \]
Consider such a linear transformation in $\R^N$ which we call $\cL$. It maps $\wt M$ to $\underbrace{W^\circ \times \dots\times W^\circ}_{l+1 \textup{~times}}$. Then it do the $K$-dilation in the direction of $\bigg(\underbrace{W^\circ \times \dots\times W^\circ}_{l+1 \textup{~times}}\bigg)^\perp$. For each $V\prec M$, we use $V'$ to denote the $m$-plane in $\Aloc(m,n)$ such that $\wt {V'}=\cL(\wt V)$. Interpreted in the coordinate of $\Aloc(l,n)$, $\cN_\de(\bV)$ becomes $\cN_{\de K}(\bV')$.

We have
\begin{equation}\label{ineq2}
\begin{split}
    \|\sum_{V\in\cV, \vec\bw(V)\in Q_{K^{-1}},V\prec M}1_{\cN_\de(\bV)}\|_{L^p(\cN_{K^{-1}}(\bM))}^p\\
    \sim K^{-(N-k)}\|\sum_{V\in\cV, \vec\bw(V)\in Q_{K^{-1}},V\prec M}1_{\cN_{\de K}(\bV')}\|_{\Aloc(l,n)}^p 
\end{split}
\end{equation} 

Next, we check the spacing condition \eqref{condition} for the set $\cV'=\cV'_M:=\{V': V\in\cV, \vec\bw(V)\in Q_{K^{-1}},V\prec M\}$. We want to prove that for any ball $Q_r\subset \Aloc(m,n)$ of radius $r$ ($\de K\le r\le 1$), we have
\[ \#\{V'\in \cV': V'\in Q_r\}\lesssim (r/\de K)^{(m+1)(d-m)+\beta}. \]
If we view $Q_r$ in the coordinate $[-1,1]^{(n-l)(l+1)}$, then there is $V_\circ=v_0(V)\times v_1(V)\times\dots\times v_l(V)$, so that $Q_r$ consists of those $V'$ so that $\wt {V'}\cap [-1,1]^N$ lie within $r$-neighborhood of $V_\circ$. After taking $\cL^{-1}$, 
\[ \{V': \wt {V'}\cap [-1,1]^N \textup{~lie~within~} r\textup{-neighborhood~of~} V_\circ\} \]
becomes
\[ \{V: \wt V\cap [-1,1]^N \textup{~lie~within~} rK^{-1}\textup{-neighborhood~of~} \cL^{-1}(V_\circ)\}. \]
Therefore,
\[ \#\{V'\in\cV':V'\in Q_r\}=\#\{V\in\cV: V\in Q_{r K^{-1}}\}, \]
where $Q_{r K^{-1}}\subset A(m,n)$ is a ball of radius $r K^{-1}$. By condition \eqref{condition}, we have
\[ \#\{V'\in\cV':V'\in Q_r\}\lesssim (r/\de K)^{(m+1)(d-m)+\beta}. \]

\bigskip

We are ready to prove \eqref{kakeyaineq}. We will induct on $\de$. Of course, \eqref{kakeyaineq} is true for small $\de$ if we choose $C_\e$ in \eqref{kakeyaineq} sufficiently large. Suppose \eqref{kakeyaineq} is true for numbers $\ge 2\de$. We will use induction hypothesis for the scale $\de K$ to bound \eqref{ineq2}. However, $\cV'$ satisfies
\begin{equation}\label{condi}
    \#\{V'\in \cV': V'\in Q_r\}\le C\cdot (r/\de K)^{(k+1)(d-k)+\beta}, 
\end{equation} 
where there is an additional constant $C$ compared with \eqref{condition}. We hope to partition $\cV'$ into subsets so that each subset satisfies the inequality above with $C=1$. This is done by the following lemma. Since the lemma is standard, we just omit the proof.

\begin{lemma}\label{goodlem}
    Let $A\subset [0,1]^m$ be a finite set that satisfies
    \[ \#(A\cap Q_r)\le M (r/\de)^s, \]
    for any $Q_r$ being a ball of radius $r$ in $[0,1]^m$ for $\de\le r\le 1$. Then we can partition $A$ into $\lesssim M^{O(1)}$ subsets $A=\sqcup A'$ such that each $A'$ satisfies
    \[ \#(A'\cap Q_r)\le (r/\de)^s, \]
    for any $Q_r$ being a ball of radius $r$ for $\de\le r\le 1$.
\end{lemma}

By the lemma, we can partition $\cV'=\bigsqcup_j \cV'_{j}$ into $O(1)$ subsets, so that \eqref{condi} holds for $\cV'_{j}$ with $C=1$. Applying the induction hypothesis \eqref{kakeyaineq} and triangle inequality, we obtain
\[ \|\sum_{V'\in \cV'}1_{\cN_{\de K}(\bV')}\|_{L^p(\Aloc(l,n))}^p\le C C_\e (\de K)^{-(m-l)(d-m)(p-1)-\e p}(\sum_{V'\in \cV'} |\cN_{\de K}(\bV')|). \]
Summing over $M\in \M_{Q_{K^{-1}}}$, we can bound the left hand side of \eqref{ineq1} by
\begin{align*}
     &\|\sum_{V\in \cV, \vec\bw(V)\in Q_{K^{-1}}}1_{\cN_{\de}(\bV)}\|_{L^p(\Aloc(l,n))}^p\\
     \le &CK^{-(N-k)} C_\e (\de K)^{-(m-l)(d-m)(p-1)-\e p}\sum_{M\in \M_{Q_{K^{-1}}}}\sum_{V'\in \cV'_M} |\cN_{\de K}(\bV')|\\
     \lesssim &C C_\e (\de K)^{-(m-l)(d-m)(p-1)-\e p}\sum_{V\in \cV, \vec\bw(V)\in Q_{K^{-1}}} |\cN_{\de}(\bV)|
\end{align*}
(The constant $C$ changes from line to line, but is independent of $K$ and $\de$.)

Summing over $Q_{K^{-1}}\in \cQ_{K^{-1}}$, we obtain
\[ \tx{I}\le C (\log K)^pK^{-\e p}  C_\e \de^{-(m-l)(d-m)(p-1)-\e p}\sum_{V\in\cV}|\cN_\de(\bV)|. \]
Combining with \eqref{estII}, we can bound the left hand side of \eqref{1term} by
\begin{align*}
\|\sum_{V\in\cV}1_{\cN_\de(\bV)}\|_{L^p(\Aloc(l,n))}^p\lesssim C (\log K)^pK^{-\e p}  C_\e \de^{-(m-l)(d-m)(p-1)-\e p}\sum_{V\in\cV}|\cN_\de(\bV)|\\
+C_{K,\e}\de^{-\e^2-(m-l)(d-m)(p-1)}\sum_{V\in\cV}|\cN_\de(\bV)|.
\end{align*} 
To close the induction, it suffices to make
\begin{equation}\label{1}
    C (\log K)^pK^{-\e p} \le 1/2 
\end{equation} 
and
\begin{equation}\label{2}
    C_{K,\e}\de^{-\e^2}\le \frac12 \de^{-\e p}C_\e.
\end{equation}
We first choose $K$ sufficiently large so that \eqref{1} holds. Then we choose $C_\e$ large enough so that \eqref{2} holds.

\section{Proof of Theorem \ref{mainthm}}\label{sec5}

We show that Theorem \ref{kakeyathm} implies Theorem \ref{mainthm} through a standard $\de$-discretization argument. Such argument was carried out in \cite[Section 3]{zahl2022unions} for a simple setup, but it easily generalizes to our setup. 

We will skip part of the technical details on the measurability which already appeared in \cite[Section 3]{zahl2022unions}, and focus on how the numerology works for general $k$.

We first localize $\cV$ to a subset of $\Aloc(m,n)$. Just like $\R^n$ can be covered by countably many unit balls, $A(m,n)$ can be covered by countably many translated and rotated copies of $\Aloc(m,n)$. We may choose one of the $\Aloc(m,n)$ so that $\dim(\cV\cap \Aloc(m,n))=\dim(\cV)=(m+1)(d-m)+\beta$. After change of coordinate, we assume this $\Aloc(m,n)$ is the one defined in Definition \ref{def3.5}. Also, we still use $\cV$ to denote $\cV\cap \Aloc(m,n)$, so now $\cV\subset \Aloc(m,n)$.

We just need to consider the case when $\beta\in (0,l+1]$. If $\beta>l+1$, we just replace $\beta$ by $l+1$ and throw away some $k$-planes in $\cV$ to make $\dim (\cV)=(m+1)(d-m)+l+1$. 
If $\beta=0$, we just replace $\beta$ by $l+1$ and $d$ by $d-1$. Next, note that it suffices to prove a variant of \eqref{mainbound}, where the right hand side is replaced by $(l+1)(d-l)+\beta-\e$ for arbitrary $\e>0$.

Fix such an $\e$. Let $\mathbf{X}=\bigcup_{V\in\cV}\bigg(A(l,V)\cap \Aloc(l,n)\bigg)$. Since $V\in\cV\subset \Aloc(m,n)$, we have
\[ \cH^{(m-l)(l+1)}\bigg(A(l,V)\cap \Aloc(l,n)\bigg)\gtrsim 1. \]
Our goal is to show \[\dim(\mathbf{X})\ge (l+1)(d-l)+\beta-\e.\]

Since $\dim(\cV)=(m+1)(d-m)+\beta$. By Frostman's lemma, there is a probability measure $\mu$ supported on $\cV$, such that for any $Q_r\subset \Aloc(m,n)$ being a ball of radius $r$, we have
\begin{equation}\label{frost}
    \mu(Q_r)\le C_0 r^{(m+1)(d-m)+\beta-\e}. 
\end{equation}

Let $k_0$ be a large integer that will be chosen later.
Cover $\mathbf{X}$ by a union $\bigcup_{k=k_0}^\infty \bigcup_{B\in\cB_k}B$, where $\cB_k$ is a collection of disjoint balls of radius $2^{-k}$ in $\Aloc(l,n)$, and 
\begin{equation}\label{lowerbound}
    \sum_k 2^{-k(\dim \mathbf{X}+\e)}\#\cB_k=:C_1<\infty.
\end{equation}
Here, each $B$ is $(l+1)(n-l)$ dimensional.

For each $V\in\cV$, there is an index $k=k(V)\ge k_0$, so that
\begin{equation}\label{lowerbound2}
    |\cH^{(m-l)(l+1)}\bigg(A(l,V)\cap \bigcup_{B\in\cB_k}B\bigg)|\gtrsim \frac{1}{k^2}. 
\end{equation} 
For each index $k$, define 
\[\cV^{(k)}=\{V\in\cV:k(V)=k\}.\] 
Then 
\[\sum_{k\ge k_0}\mu(\cV^{(k)})=1, \]
so there exists an index $k_1\ge k_0$ with $\mu(\cV^{(k_1)})\ge \frac{1}{k_1^2}$. From now on, we just define $\de=2^{-k_1}$ and $\mathbf{E}_\de=\bigcup_{B\in\cB_{k_1}}B.$

From \eqref{lowerbound}, we see that

\begin{equation}\label{lowerbd}
    |\mathbf{E}_\de|=\#\cB_{k_1}2^{-k_1(l+1)(n-l)}\le C_1 \de^{(l+1)(n-l)-\e-\dim \mathbf{X}}.
\end{equation}

Noting that $\mu$ is a $\big((k+1)(d-k)-\e\big)$-dimensional Frostman measure (see \eqref{frost}), by \cite[Lemma 8]{dote2022exceptional}, there exists a subset $\cV_1\subset \cV^{(k_1)}$ that satisfies $\# \cV_1\gtrsim \frac{1}{k_1^2}\de^{\e-(m+1)(d-m)-\beta}$, and
\[ \#\{V\in\cV_1: V\in Q_r\}\lesssim \de^{-\e}(r/\de)^{(m+1)(d-m)+\beta} \]
for any $Q_r\subset \Aloc(m,n)$ being a ball of radius $r$ ($\de\le r\le 1$).

In what follows, we will write $A\lessapprox B$ to mean $A\lesssim |\log\de|^{O(1)} B$.

By Lemma \ref{goodlem}, we can find $\cV_2\subset \cV_1$ so that $\# \cV_2\gtrsim \de^{O(\e)-\big((m+1)(d-m)+\beta\big)}$, and $\cV_2$ satisfies
\[ \#\{V\in\cV_2: V\in Q_r\}\le (r/\de)^{(k+1)(d-k)+\beta}. \]
From \eqref{lowerbound2}, we see that for any $V\in\cV_2$, 
\begin{equation}\label{lowerbd2}
    |\cN_\de(\bV)\cap \mathbf{E}_\de|\gtrapprox  |\cN_\de(\bV)|.
\end{equation} 

Now, we can apply Theorem \ref{kakeyathm} to $\cV_2$. 
By \eqref{kakeyaineq} and H\"older's inequality, we have
\[ \big(\int_{\mathbf{E}_\de} \sum_{V\in\cV_2} 1_{\cN_\de(\bV)}\big)\big| 
\mathbf{E}_\de\big|^{-1/p'}\le \| \sum_{V\in\cV_2} 1_{\cN_\de(\bV)}\|_p\le C_\e \de^{-\frac{(m-l)(d-m)}{p'}-\e}\big(\sum_{V\in\cV_2}|\cN_\de(\bV)|\big)^{1/p}.    \]
Noting \eqref{lowerbd2}, this implies 
\[ | \mathbf{E}_\de |\gtrsim \de^{(m-l)(d-m)+p'\e}\big(\sum_{V\in\cV_2}|\cN_\de(\bV)|\big). \]
Combined with $\sum_{V\in\cV_2}|\cN_\de(\bV)|=\#\cV_2 \de^{(l+1)(n-m)}\gtrsim \de^{O(\e)-(m+1)(d-m)-\beta} \de^{(l+1)(n-m)}$ and \eqref{lowerbd}, we obtain that
\[ \dim \mathbf{X}\ge (l+1)(d-l)+\beta-O(\e) \]
as $\de$ goes to $0$.

\newpage

\appendix
\section{\texorpdfstring{$k$}{}-geodesics and projections in Grassmannians}

In this section, we show some results realated to geodesics and projections in Grassmannians. The tool we use is called the singular value decomposition. Since it is fundamental in linear algebra, we may be sketchy in some steps.
For more details on singular value decomposition, we refer to \cite{golub2013matrix}. See also the discussions in \cite{conway1996packing}. 

We first recall the singular value decomposition. For $1\le l\le m$, let $A$ be any $l\times m$ real matrix. Then there exist orthogonal matrices $\Omega_1$ and $\Omega_2$ such that
\[ \Omega_1 A \Omega_2 =\begin{bmatrix}
    \lambda_1 & & &\\
    & \ddots & &\\
    & & \lambda_r &\\
    & & & 0
\end{bmatrix}, \]
where $\lambda_1\ge \lambda_2\ge\dots\ge \lambda_r>0$, and $\lambda_1^2,\dots,\lambda_r^2$ are all the nonzero eigenvalues of $A A^T$.

Given integers $1\le l<n$, we use $G(l,n)$ to denote the set of $l$-dimensional subspaces in $\R^{n}$. The natural isomorphism $G(l,n)\cong O(n)/O(l)\times O(n-l)$ endowed $G(l,n)$ with a Riemannian metric $g_\circ$ which is invariant under the action of $O(n)$. 

The distance function induced by $g_\circ$ is characterized using principle angles. Given $V, W\in G(l,n)$, we choose an orthonormal basis $\{\wt v_i\}$ for $V$ and an orthonormal basis $\{\wt w_j\}$ for $W$. Consider the $l\times l$ matrix $A$ whose $(i,j)$-entry is $v_i\cdot w_j$. In other words,
\[ A=[\wt v_1\  \wt v_2\ \cdots\ \wt v_l]\left[ \begin{array}{cc}
     \wt w_1^T  \\
     \wt w_2^T\\
     \vdots\\
     \wt w_l^T 
\end{array} \right]. \]
Suppose the singular value decomposition for $A$ is
\[ \Omega_1 A \Omega_2 =\begin{bmatrix}
    \lambda_1 & & &\\
    & \ddots & &\\
    & & \lambda_r &\\
    & & & 0
\end{bmatrix}. \]
Then, we define the principal angles between $V$ and $W$ as 
\[\theta_1=\arccos \lambda_1,\dots,\theta_{r}=\arccos\lambda_r, \theta_{r+1}=\dots=\theta_l=0.\]

Actually, the principal angles between $V$ and $W$ can also be defined recursively as
\[\cos\theta_i=\max_{v\in V}\max_{w\in W}v\cdot w=v_i\cdot w_i,\]
for $i=1,\dots,l+1$, subject to $\|v\|=\|w\|=1$, $v\cdot v_j=w\cdot w_j=0$ for $j=1,\dots,i-1$. Moreover, $\{v_i\}$ and $\{w_i\}$ are orthonormal bases of $V$ and $W$ respectively. Also, $v_i\cdot w_j=\cos\theta_i$ if $i=j$, $v_i\cdot w_j=0$ if $i\neq j$. 
One actually sees that 
\[ [v_1\ v_2\ \dots\ v_l]=\Omega_1[\wt v_1\ \wt v_2\ \dots\ \wt v_l],\ [w_1\ w_2\ \dots\ w_l]=\Omega_1[\wt w_1\ \wt w_2\ \dots\ \wt w_l]. \]

Suppose the principal angles between $V$ and $W$ are $\theta_1,\dots,\theta_l$, then the distance between $V$ and $W$ is given by
\[ d_{g_\circ}(V,W)=\sqrt{\theta_1^2+\dots+\theta_{l}^2}. \]

Using $\{v_i\}$ and $\{w_j\}$, we can explicitly construct a geodesic connecting $V$ and $W$. Let $\ga: [0,1]\rightarrow G(l,n)$ defined by 
\[\ga(t)=\text{span}\{\cos( t\theta_1) v_1+\sin( t\theta_1) v_1^\perp,\dots, \cos( t\theta_{l}) v_{l}+\sin( t\theta_{l}) v_{l}^\perp\}.\]
Here, each $v_i^\perp$ lies in $\text{span}\{v_i,w_i\}$ and is perpendicular to $v_i$, so that $\cos(\theta_i)v_i+\sin(\theta_i)v_i^\perp=w_i$. From the construction, we see that $\ga(0)=V, \ga(1)=W$. One can also check $d_{g_\circ}(V,\ga(t))=t\cdot d_{g_\circ}(V,W)$.

Now, we easily deduce the following result.
\begin{lemma}\label{lemm1}
    Let $\Pi\in G(m,n)$, where $l\le m\le n$. Then, $G(l,\Pi)$ is a totally geodesic submanifold of $(G(l,n), g_\circ)$.
\end{lemma}
\begin{proof}
    We just note that given $V, W\in G(l,n)$ such that $V,W\subset \Pi$, then the geodesic connecting them is also contained in $\Pi$.
\end{proof}

Next, we talk about the projections. Let $\Pi\in G(m,n)$, where $l\le m\le n$.  We already proved that $G(l,\Pi)$ is a totally geodesic submanifold of $G(l,n)$. For simplicity, we denote $G(l,\Pi)$ by $\mathbf{\Pi}$. We want to describe the projection $\pi_\mathbf{\Pi}: G(l,n)\rightarrow \mathbf{\Pi}$. Recall the definition of $\pi_{\mathbf{\Pi}}$: given $V\in G(l,n)$, then $W=\pi_{\mathbf{\Pi}}(V)$ satisfies $d_{g_\circ}(V,W)=d_{g_\circ}(V,\mathbf{\Pi})$. If such $W$ is not be unique, then we just choose one of them. We have

\begin{lemma}\label{lemm2}
     Let $\Pi\in G(m,n),$ where $l\le m\le n$. Then for any $V\in G(l,n)$, $\pi_\Pi(V)\subset\pi_{\mathbf{\Pi}}(V)$. Here, $\pi_\Pi:\R^{n}\rightarrow \Pi$ is the standard orthogonal projection.
\end{lemma}

\begin{proof}[Sketch of proof]
    Actually, we can also use the principal angles to describe $d_{g_\circ}(V,\mathbf{\Pi})$. The principle angles $\theta_1,\dots,\theta_{l}\in [0,\frac{\pi}{2}]$ between $V$ and $\Pi$ are recursively defined by 
    \[\cos\theta_i=\max_{v\in V}\max_{w\in \Pi}v\cdot w=v_i\cdot w_i,\]
for $i=1,\dots,l$, subject to $\|v\|=\|w\|=1$, $v\cdot v_j=w\cdot w_j=0$ for $j=1,\dots,i-1$.
Then,
\[ d_{g_\circ}(V,\mathbf{\Pi})=\sqrt{\theta_1^2+\dots+\theta_{l}^2}. \]
Moreover, $\{v_i\}$ and $\{w_i\}$ are orthonormal. If we denote $W=\textup{span}\{w_i\}$, then $d_{g_\circ}(V,W)=d_{g_\circ}(V,\mathbf{\Pi})$.

We want to show the $W$ obtained as above is contained in $\pi_\Pi(V)$. Actually, we will show $w_i\parallel \pi_\Pi(v_i)$. (We will use the convention the $\vec {0}$ is parallel to any vector.) Let $\Pi_i$ be the orthogonal complement of $\textup{span}\{w_1,\dots,w_{i-1}\}$ in $\Pi$.  In other words,
\[ \Pi=\textup{span}\{w_1,\dots,w_{i-1}\}\perp \Pi_i. \]
Then we have
\[ v_i\cdot w_i=\max_{w\in \Pi_i, \|w\|=1} v_i\cdot w.  \]
To maximize the dot product, we must have
\[ w_i\parallel \pi_{\Pi_i}(v_i). \]
Since $v_i$ is orthogonal to $\textup{span}\{w_1,\dots,w_{i-1}\}$, we have
\[ \pi_{\Pi_i}(v_i)=\pi_{\Pi}(v_i). \]
Therefore, we have
\[ w_i\parallel \pi_\Pi(v_i). \]
This shows that
\[ \pi_\Pi(V)=\textup{span}\{\pi(v_i)\}\subset \textup{span}\{w_i\}=W=\pi_{\mathbf{\Pi}}(V). \]
\end{proof}

\bigskip

We turn to $A(l,n)$.
For any $L\in G(l+1,n+1)$ that is not parallel to $\R^n\times \{0\}$, we see that $L\cap (\R^n\times \{1\})$ is an $l$-plane. Therefore, we identify a generic open subset of $G(l+1,n+1)$ with $A(l,\R^n\times \{0\})$. We just simply denote $A(l,\R^n\times \{0\})$ by $A(l,n)$. Using the embedding 
\[ A(l.n)\hookrightarrow G(l+1,n+1), \]
the Riemannian metric $g_\circ$ pulls back to a Riemannian metric on $A(l,n)$ for which we denote by $\wt g_\circ$. 

It is straightforward to see the following results which are reinterpretations of Lemma \ref{lemm1}, Lemma \ref{lemm2}.

\begin{lemma}\label{lemmm1}
    Let $\Pi\in A(m,n)$, where $l\le m\le n$. Then, $A(l,\Pi)$ is a totally geodesic submanifold of $(A(l,n), \wt g_\circ)$.
\end{lemma}

\begin{lemma}\label{lemmm2}
Let $\Pi\in G(m,n)$, where $l\le m\le n$. Then for any $L\in A(l,n)$, $\pi_{\Pi}(L)\subset\pi_{A(l,\Pi)}(L)$. 
\end{lemma}

 Here, $\pi_\Pi:\R^{n}\rightarrow \Pi$ is the standard orthogonal projection. $\pi_{A(l,\Pi)}: A(l,n)\rightarrow A(l,\Pi)$ is the projection in $A(l,n)$. We also remark that if we replace $\Pi\in G(m,n)$ by $\Pi\in A(m,n)$, the result is not true.

\bibliographystyle{abbrv}
\bibliography{bibli}

\end{document}